\documentclass[12pt]{amsart}

\usepackage{amsfonts}
\usepackage{amssymb, latexsym, amsthm, amsmath, verbatim}
\usepackage{indentfirst}

\usepackage{hyperref}

\numberwithin{equation}{section}
\theoremstyle{definition}
\newtheorem{Thm}{Theorem}[section]
\newtheorem{Prop}[Thm]{Proposition}
\newtheorem{Cor}[Thm]{Corollary}

\newtheorem{Rmk}[Thm]{Remark}

\allowdisplaybreaks[4]

\usepackage{mathtools}
\usepackage{tensor}
\usepackage{mathrsfs}

\newtheorem*{main}{Main Theorem}

\makeatletter
\def\imod#1{\allowbreak\mkern5mu{\operator@font mod}\,\,#1}
\makeatother

\begin{document}

\title[Rationality]{Rationality of the Petersson Inner Product of Cohen's Kernels}
\author{Yuanyi You and Yichao Zhang$^*$}
\address{Department of Mathematics, Harbin Institute of Technology, Harbin 150001, P. R. China}
\email{1546711213@qq.com}
\address{Institute for Advanced Study in Mathematics of HIT and School of Mathematics, Harbin Institute of Technology, Harbin 150001, P. R. China}
\email{yichao.zhang@hit.edu.cn}
\date{}
\thanks{$^{\star}$ partially supported by a grant of National Natural Science Foundation of China (no. 11871175)}

\begin{abstract} By explicitly calculating and then analytically continuing the Fourier expansion of the twisted double Eisenstein series $E_{s,k-s}^{*}(z,w; 1/2)$ of Diamantis and O'Sullivan, we prove a formula of the Petersson inner product of Cohen's kernel and one of  its twists, and obtain a rationality result. This extends a result of Kohnen and Zagier.
\end{abstract}

\maketitle

\newcommand{\Z}{{\mathbb Z}} 
\newcommand{\Q}{{\mathbb Q}} 

\section{Introduction and Statement of the Main Theorem}
Let $\mathcal{S}_k$ be the complex vector space of holomorphic cusp forms of even integral weight $k$ for $\textrm{SL}_2(\mathbb{Z})$, and let $\mathcal{B}_k$ be the basis of normalized Hecke eigenforms in $\mathcal{S}_k$. For $ f(\tau)=\sum_{n\ge1}a_f(n)e^{2\pi i n \tau}\in\mathcal{S}_k$, denote its $L$-function by $L(f,s)=\sum_{n\ge1}\frac{a_f(n)}{n^s}$ and its completed $L$-function by
\[
L^*(f,s)=(2\pi)^{-s}\Gamma(s)L(f,s).
\]
It is well-known that $L^*(f,s)$ is entire and satisfies the functional equation $L^*(f,k-s)=(-1)^{k/2}L^*(f,s)$. For integers $n$ with $0\le n\le k-2$, the $n$-th period of $f$ is
\[
r_n(f):=L^*(f,n+1).
\]
One of the most important results on these periods is Manin's Periods Theorem \cite{manin}, which concerns the algebraicity of such periods of normalized Hecke eigenforms.

For each $0\le n\le k-2$, since $f\mapsto r_n(f)$ is linear, with the Petersson inner product, there must exist $R_n\in \mathcal S_k$ such that $\langle f, R_n\rangle=r_n(f)$ for all $f\in \mathcal S_k$. Kohnen and Zagier \cite{Kohnen1983} showed that, for $m,n$ of opposite parity, $\langle R_m, R_n\rangle$ is a rational number by proving an explicit formula. We shall prove a twisted version of such formula and then obtain the corresponding rationality.

To state out main theorem, let $\mathcal D_k(z,s)$, $s\in\mathbb{C}$, be the unique cusp form that satisfies that
\[
\langle{\mathcal D}_k(z,s), f\rangle=L^*(\overline{f},s)
\]
for all $f\in \mathcal S_k$ (see \cite{2010Kernels}). Clearly, $R_n={\mathcal D}_k(\cdot,n+1 )$. The Cohen kernel is
\[
\mathcal{C}_k(z,s)=\frac{1}{2}\sum_{\gamma\in \textrm{SL}_2(\mathbb Z)}(\gamma z)^{-s} j(\gamma,z)^{-k},
\]
and
\[
{\mathcal{C}}_{k}(z, s)=2^{2-k}(-1)^{k/2}\pi e^{-\pi i s/2}\frac{\Gamma(k-1)}{\Gamma(s)\Gamma(k-s)}{\mathcal D}_k(z,s).
\]
The twisted Cohen kernel (see \cite{2013Kernels}), for $p/q\in\mathbb{Q}$, is defined to be
\[
\mathcal{C}_k(z,s; p/q):=\frac{1}{2}\sum_{\gamma\in\Gamma}(\gamma z+p/q)^{-s} j(\gamma,z)^{-k},
\]
and $\mathcal{C}_k(z,s; 0)=\mathcal{C}_k(z,s)$.

\begin{main}
For integers $2\leq m,n\leq k-2$ of opposite parity,
\[
\pi^{-2} e^{\pi i (m-n)/2}\langle{\mathcal{C}}_{k}(z, m; 1/2),{\mathcal{C}}_{k}(z, n)\rangle\in \mathbb Q.
\]
\end{main}

Note that the extra factor $\pi^{-2} e^{\pi i (m-n)/2}$ comes from the difference of $\mathcal{C}_k(z,s)$ and $\mathcal{D}_k(z,s)$ and that of their twisted counterparts.

We follow the lines in \cite{CHOIE2020} and compute the Fourier expansion of the twisted double Eisenstein series in Section 2. In Section 3, we carry out its analytic continuation, from which we derive the desired formula and the rationality in the final section by looking at the first Fourier coefficient.

\medskip

\noindent {\bf Acknowledgement}:
The second author was supported by NSFC11871175. The authors are very grateful to the referee for providing many valuable comments on a previous version of this paper and especially for pointing out an error in the formula of Proposition 4.1.

\bigskip

\section{Fourier Expansion of the Twisted double Eisenstein series}
For $p/q\in\mathbb{Q}$ with $q>0$ and even integer $k\geq 6$, Diamantis and O'Sullivan \cite{2013Kernels} introduced
the twisted double Eisenstein series $E_{s,k-s}(z,w; p/q)$ by
\begin{equation}\label{Formula-2.1}
\begin{aligned}
&\zeta(1-w+s)\zeta(1-w+k-s)E_{s,k-s}(z,w; p/q)\\
=&\sum_{a,b,c,d\in \mathbb Z, ad-bc>0}(ad-bc)^{w-1}\left(\frac{az+b}{cz+d}+\frac{p}{q}\right)^{-s}(cz+d)^{-k}.\\
\end{aligned}
\end{equation}
In \cite{2013Kernels}, it is shown that $E_{s,k-s}(z,w; p/q)$ converges absolutely and uniformly on compact subsets in $\mathcal D$, where
\begin{align*}
\mathcal D=\{2<\text{Re}(s)<k-2,\ \  \text{Re}(w)<\min\{ \text{Re}(s)-1, k-\text{Re}(s)-1\}         \}.
\end{align*}
The completed twisted double Eisenstein series is defined by
\begin{equation}\label{Formula-2.2}
\begin{aligned}
&E_{s,k-s}^{*}\left(z,w; p/q\right)\\=&\frac{e^{\pi i s/2}\Gamma(s)\Gamma(k-s)\Gamma(k-w)\zeta(1-w+s)\zeta(1-w+k-s)}{2^{3-w}\pi^{k+1-w}\Gamma(k-1)}E_{s,k-s}\left(z,w; p/q\right).
\end{aligned}
\end{equation}
They showed that
 $E_{s,k-s}^{*}(z,w; p/q)$ has an analytic continuation to all $s, w\in \mathbb C$, as a function of $z$ it always belongs to $S_{k}$, and for any $f\in {\mathcal B}_k$
\[
 \langle E_{s,k-s}^{*}(\cdot,w; p/q), f\rangle =L^*(f,k-s;p/q)L^*(f,k-w).
\]
It follows that
\begin{equation}\label{Formula-2.3}
\begin{aligned}
&E_{s,k-s}^*(z,w; p/q)
=\sum_{f\in {\mathcal B}_k}\frac{L^*(f,k-s;p/q)L^*(f,k-w)}{\langle f,f\rangle}f(z),
\end{aligned}
\end{equation}
so that the left-hand side has analytic continuation to all of $(s,w)\in{\mathbb C}^2$. Moreover, it satisfies two functional equations
\begin{align*}
E_{s,k-s}^{*}(z,k-w; p/q)&=(-1)^{k/2}E_{s,k-s}^{*}(z,w; p/q),\\
q^sE_{k-s, s}^{*}(z,w; p/q)&=(-1)^{k/2}q^{k-s}E_{s,k-s}^{*}(z,w; -p'/q),
\end{align*}
for $p'p\equiv 1\mod q$.
\begin{Prop}\label{Prop-2.1}(Fourier expansion) Let $(s, w)\in \mathcal D$, then $E_{s,k-s}^{*}(z,w; 1/2)$ has the Fourier expansion
\[
E_{s,k-s}^{*}\left(z,w; 1/2\right)=\frac{\Gamma(s)\Gamma(k-s)\Gamma(k-w)}{2^{2-s-w}\pi^{k+1-w}\Gamma(k-1)}\sum_{m=1}^{\infty}c(m)e^{2\pi i mz},
\]
where the $m$-th Fourier coefficient $c(m)=c_m(s,w)$ is equal to the following absolute convergent series
\begin{align*}
&\frac{(2\pi)^{s}}{2^s\Gamma (s)}m^{s-1}\zeta(k-s-w+1)\sum_{2a|m}a^{w-s}+\frac{(-1)^{\frac{k}{2}}(2\pi)^{k-s}}{2^{k-s}\Gamma (k-s)}m^{k-s-1}\zeta(s-w+1)\sum_{2a|m }a^{s+w-k}\\
+&\frac{(2\pi)^{k-w}\Gamma(s+w-k)}{2^{s-1}\Gamma (s)}m^{s-1}\cos(\pi(s+w-k)/2)\zeta\left(s+w-k,1/2\right)\sum_{a\mid m, 2a\nmid m}a^{w-s}\\
+&(-1)^{\frac{k}{2}}\frac{(2\pi)^{k-w}\Gamma(w-s)}{2^{k-s-1}\Gamma (k-s)}m^{k-s-1}\cos(\pi(w-s)/2)\zeta\left(w-s,1/2\right)\sum_{a\mid m, 2a\nmid m }a^{s+w-k}\\
+&\frac{(-1)^{\frac{k}{2}}(2\pi)^{k} m^{k-1}}{2^s\Gamma(s)\Gamma (k-s)}\sum_{a+c/2>0,c>0, (a,c)=1}c^{s-k}(a+c/2)^{-s}\sum_{n>0}n^{w-1}\sum_{r|m}r^{w-k}\\
&\ \ \times \left(e^{\pi i s/2}e^{2\pi i \frac{m}{r}\frac{na'}{c}}{{}_1{f}_1}\left(s,k;-\frac{2\pi i m n }{rc(a+c/2)}\right)+e^{-\pi i s/2}e^{-2\pi i \frac{m}{r}\frac{na'}{c}}{{}_1{f}_1}\left(s,k;\frac{2\pi i m n }{rc(a+c/2)}\right)\right).
\end{align*}
Here all the sums are over positive integers,  $a'\equiv a^{-1} \ (\text{mod}\ c)$, and
\[
{{}_1{f}_1}(\alpha,\beta;z)=\frac{\Gamma(\alpha)\Gamma(\beta-\alpha)}{\Gamma(\beta)}{{}_1{F}_1}(\alpha, \beta; z)
\]
with ${{}_1{F}_1}(\alpha, \beta; z)$ being Kummer's degenerate hypergeometric function.
\end{Prop}
\begin{proof}
The proof is similar to that of Proposition 4.2 in \cite{CHOIE2020} and we omit some details in the following.

From \eqref{Formula-2.1} and \eqref{Formula-2.2}, we have
\begin{equation}\label{Formula-2.4}
\begin{aligned}
\frac{e^{\pi i s/2}}{2^{1+s}}\sum_{a,b,c,d\in \mathbb Z, ad-bc>0}(ad-bc)^{w-1}\left(\frac{az+b}{cz+d}+\frac{1}{2}\right)^{-s}(cz+d)^{-k}=
\sum_{m \ge 1}c(m)e^{2\pi i mz}.
\end{aligned}
\end{equation}
To compute $c(m)$, we split the summation of left-hand side of \eqref{Formula-2.4} into four cases and consider the contribution of each to $c(m)$.

(\textbf{1}) For the elements $\gamma=\left[ \begin{matrix}
   a_{\gamma} & b_{\gamma}  \\
   c_{\gamma} & d_{\gamma} \\
\end{matrix} \right] \in  M_ n $
with $c_{\gamma}=0$, $n\geq 1$, the corresponding sub-series is
\begin{align*}
I_1=&\frac{e^{\pi i s/2}}{2^{1+s}}\sum_{n>0}n^{w-1}\sum_{a,b,c,d\in \mathbb Z, ad-bc=n,c=0}\left(\frac{az+b}{cz+d}+\frac{1}{2}\right)^{-s}(cz+d)^{-k}\\
=&2^{-s}e^{\pi i s/2}\sum_{n>0}n^{w-1}\sum_{a,d\in \mathbb Z, ad=n, a>0}d^{s-k}\sum_{b\in \mathbb Z}\left(az+\frac{d}{2}+b\right)^{-s}.
\end{align*}
Applying Lipschitz's formula
\begin{equation}\label{Formula-2.5}
\begin{aligned}
\sum_{n\in\mathbb Z }(\tau+n)^{-s}=\frac{e^{-\pi i s/2}(2 \pi)^s}{\Gamma(s)}\sum_{n\ge1}n^{s-1}e^{2\pi i n \tau}, \ \ \text{Im}(\tau)>0, \text{Re}(s)>1,
\end{aligned}
\end{equation}
to the sum over $b$, we have
\[
I_1=\frac{(2\pi)^{s}}{2^s\Gamma (s)}\sum_{n>0}n^{w-1}\sum_{a,d\in \mathbb Z, ad=n, a>0}d^{s-k}\sum_{r>0}r^{s-1}e^{2\pi i rd/2}e^{2\pi i raz}.
\]
Then the contribution of $I_1$ in the $m$-th Fourier coefficient is equal to
\begin{align*}
\nonumber c(m)_1=&\frac{(2\pi)^{s}}{2^s\Gamma (s)}\sum_{n>0}n^{w-1}\sum_{a,d\in \mathbb Z, ad=n, a>0}d^{s-k}\sum_{ra=m}r^{s-1}e^{2\pi i rd/2}\\
\nonumber=&\frac{(2\pi)^{s}}{2^s\Gamma (s)}\sum_{n>0}n^{w-1}\sum_{ ad=n, a|m}d^{s-k}\left(\frac{m}{a}\right)^{s-1}e^{2\pi i md/2a}\\
\nonumber=&\frac{(2\pi)^{s}}{2^s\Gamma (s)}\sum_{a|m}\sum_{d\ge1}(ad)^{w-1}d^{s-k}\left(\frac{m}{a}\right)^{s-1}e^{2\pi i md/2a}\\
\nonumber=&\frac{(2\pi)^{s}}{2^s\Gamma (s)}m^{s-1}\sum_{a|m}a^{w-s}F(k-s-w+1, m/2a ),
\end{align*}
where $F(s,a)=\sum_{n=1}^\infty \frac{e^{2\pi ian}}{n^s}$.
By its connection with Hurwitz zeta function (see Formula 25.13.2 of \cite{Olver2010NIST}),
\begin{align*}
c(m)_1=&\frac{(2\pi)^{s}}{2^s\Gamma (s)}m^{s-1}\zeta(k-s-w+1)\sum_{2a|m}a^{w-s}\\
&+\frac{(2\pi)^{k-w}\Gamma(s+w-k)}{2^{s-1}\Gamma (s)}m^{s-1}\cos(\pi(s+w-k)/2)\zeta\left(s+w-k,1/2\right)\sum_{a\mid m, 2a\nmid m}a^{w-s}.
\end{align*}

(\textbf{2})  For the elements $\gamma=\left[ \begin{matrix}
   a_{\gamma} & b_{\gamma}  \\
   c_{\gamma} & d_{\gamma} \\
\end{matrix} \right] \in  M_ n $
with $a_{\gamma}+c_{\gamma}/2=0$, $n\geq 1$, the corresponding subseries of (\ref{Formula-2.4}) is given by
\begin{align*}
I_2=&\frac{e^{\pi i s/2}}{2^{1+s}}\sum_{n>0}n^{w-1}\sum_{a,b,c,d\in \mathbb Z, ad-bc=n,a+c/2=0}\left(\frac{az+b}{cz+d}+\frac{1}{2}\right)^{-s}(cz+d)^{-k}\\
=&\frac{e^{\pi i s/2}}{2^{s}}\sum_{n>0}n^{w-1}\sum_{a,b,c,d\in \mathbb Z, ad-bc=n,a+c/2=0, a>0}\left(\frac{b+d/2}{cz+d}\right)^{-s}(cz+d)^{-k}.
\end{align*}
Note that the summation in $I_1$ and that in $I_2$ do not overlap, since otherwise $c_{\gamma}=0$ and $a_{\gamma}+c_{\gamma}/2=0$ and hence $a_{\gamma}d_{\gamma }-b_{\gamma}c_{\gamma}=0$. In particular, $c_{\gamma}, a_{\gamma}\neq 0$. It is clear that the conditions on $\gamma$ in the last series are equivalent to
\[a>0, \quad a\mid n,\quad  c=-2a,\quad b\in\mathbb{Z},\quad d = \frac{n}{a}-2b.\]
Therefore
\begin{align*}
I_2=&\frac{e^{\pi i s/2}}{2^{s}}\sum_{n>0}n^{w-1}\sum_{a>0, a\mid n }\left(\frac{n}{2a}\right)^{-s}2^{s-k}e^{-\pi i(s-k)}\sum_{b\in \mathbb Z}\left(az-\frac{n}{2a}+b\right)^{-(k-s)}.
\end{align*}
So use the same method of $I_1$ and apply Lipschitz's formula (\ref{Formula-2.5}) to the sum over $b$, and the contribution of $I_2$ to $c(m)$ is equal to
\begin{align*}
c(m)_2=&(-1)^{k/2}\frac{(2\pi)^{k-s}}{2^{k-s}\Gamma(k-s)}\sum_{n>0}n^{w-s-1}\sum_{a>0,a\mid (n,m)}a^s(m/a)^{k-s-1}e^{-\pi i nm/a^2}\\
=&(-1)^{k/2}\frac{(2\pi)^{k-s}}{2^{k-s}\Gamma(k-s)}\sum_{a\mid m}\sum_{n>0}(an)^{w-s-1}a^{2s-k+1}m^{k-s-1}e^{-2\pi i mn/(2a)}\\
=&(-1)^{k/2}\frac{(2\pi)^{k-s}}{2^{k-s}\Gamma(k-s)}m^{k-s-1}\sum_{a\mid m}a^{w+s-k}F(s-w+1, -m/(2a))\\
=&\frac{(-1)^{k/2}(2\pi)^{k-s}}{2^{k-s}\Gamma (k-s)}m^{k-s-1}\zeta(s-w+1)\sum_{2a|m}a^{s+w-k}\\
&+(-1)^{k/2}\frac{(2\pi)^{k-w}\Gamma(w-s)}{2^{k-s-1}\Gamma (k-s)}m^{k-s-1}\cos(\pi(w-s)/2)\zeta\left(w-s,1/2\right)\sum_{a\mid m, 2a\nmid m }a^{s+w-k}.
\end{align*}

(\textbf{3}) For the elements $\gamma=\left[ \begin{matrix}
   a_{\gamma} & b_{\gamma}  \\
   c_{\gamma} & d_{\gamma} \\
\end{matrix} \right] \in  M_ n $
with $(a_{\gamma}+c_{\gamma}/2)c_{\gamma}>0$, the corresponding sub-series in (\ref{Formula-2.4}) is given by
\begin{align*}
I_{3}=&\frac{e^{\pi i s/2}}{2^{1+s}}\sum_{n>0}n^{w-1}\sum_{a,b,c,d\in \mathbb Z, ad-bc=n,(a+c/2)c>0}\left(\frac{az+b}{cz+d}+\frac{1}{2}\right)^{-s}(cz+d)^{-k}.
\end{align*}
The set of integral matrices with determinant $n$ can be listed as follows:
\[
 M_{n}=\left\{{\left( \begin{matrix}
   ar & nb_0/r+(t+rl)a  \\
   cr & nd_0/r+(t+rl)c \\
\end{matrix} \right) : r|n,\ \textrm {gcd}(a,c)=1,\ t \in \mathbb Z / r \mathbb Z,\  l \in \mathbb Z }\right\},
\]
where for each pair $(a, c)$, $b_0, d_0$  are fixed so that $ad_0-b_0c=1$. Using the same method for $\text{III}$ of \cite{CHOIE2020}, we see that its contribution $c(m)_3$
to $c(m)$ is equal to
\begin{align*}
&\frac{e^{\pi i s/2}}{2^{1+s}}\sum_{n>0}n^{w-1}\sum_{(a+c/2)c>0, (a,c)=1, r|n}\sum_{t\in \mathbb Z/r\mathbb Z}\sum_{l\in \mathbb Z}\\
&\times\int_{iC}^{iC+1}\left(\frac{ra(z+l)+nb_0/r+ta }{rc(z+l)+nd_0/r+tc}+\frac{1}{2}\right)^{-s}(rc(z+l)+nd_0/r+tc)^{-k}e^{-2\pi i m(z+l)}dz\\
=&\frac{e^{\pi i s/2}}{2^{1+s}}\sum_{n>0}n^{w-1}\sum_{(a+c/2)c>0, (a,c)=1}\sum_{r|(m,n)}re^{2\pi i \frac{m}{r}\frac{n}{r}\frac{a'}{c}}\\
&\times\int_{iC-\infty}^{iC+\infty}\left(\frac{a}{c}-\frac{n}{c^2r^2z} +\frac{1}{2}\right)^{-s}
 (rcz)^{-k}e^{-2\pi i mz}dz,
\end{align*}
where we choose $0<a'\le |c|$ and $aa'-b_0c=1$, the shift by $l$ makes up the integral $\int_{iC-\infty}^{iC+\infty}$, and a change of variable $z\rightarrow z-(nd_0+trc)/(r^2c)$ gives the last expression. By change of variable
$z\to \frac{c}{a+c/2}iz$ and the integral representation of ${}_1F_1$, we have
\begin{align*}
c(m)_3=&\frac{e^{\pi i s/2}}{2^{1+s}}\sum_{n>0}n^{w-1}\sum_{(a+c/2)c>0, (a,c)=1}\sum_{r|(m,n)}r^{1-k}c^{-k}e^{2\pi i \frac{m}{r}\frac{n}{r}\frac{a'}{c}}\\
&\times\int_{iC-\infty}^{iC+\infty}\left(\left(\frac{a}{c}+\frac{1}{2}\right)z-\frac{n}{c^2r^2} \right)^{-s}z^{s-k}e^{-2\pi i mz}dz\\
=&e^{\pi i s/2}(-1)^{\frac{k}{2}}\frac{(2\pi)^{k} m^{k-1}}{2^{s}\Gamma(k)}\sum_{a+c/2>0, c>0, (a,c)=1 }c^{s-k}(a+c/2)^{-s}\sum_{n>0}n^{w-1}\\
&\times \sum_{r|m}r^{w-k}e^{2\pi i \frac{m}{r}\frac{na'}{c}}{{}_1{F}_1}\left(s,k;-\frac{2\pi i m n }{rc(a+c/2)}\right).
\end{align*}

(\textbf{4}) Finally the computation on terms with $(a+c/2)c < 0$ reduces to that of part (\textbf{3}), by noting that
\begin{small}
\begin{align*}
\left(\frac{a}{c}+\frac{1}{2}-\frac{n}{c^2r^2z}\right)^{-s}=e^{-\pi i s }z^{s}\left(-\left(\frac{a}{c}+\frac{1}{2}\right)z+\frac{n}{c^2r^2}\right)^{-s}.
\end{align*}
\end{small}
The contribution is equal to
\begin{align*}
{c(m)}_{4}
=&(-1)^{\frac{k}{2}}\frac{(2\pi)^{k} m^{k-1}}{e^{\pi i s/2 }2^{s}\Gamma(k)}\sum_{a+c/2>0,c>0, (a,c)=1}c^{s-k}(a+c/2)^{-s}\sum_{n>0}n^{w-1}\\
&\times \sum_{r|m}r^{w-k}e^{-2\pi i \frac{m}{r}\frac{na'}{c}}{{}_1{F}_1}\left(s,k;\frac{2\pi i m n }{rc(a+c/2)}\right).
\end{align*}
Putting together the formulas for $c(m)_1,c(m)_2,c(m)_3$ and $c(m)_4$ obtained above, we obtain the desired formula for $c(m)$.
\end{proof}

\begin{Cor}\label{Cor-2.2}
The following identity
\begin{align*}
&\frac{2^{2-s-w}\pi^{k+1-w}\Gamma(k-1)}{\Gamma(s)\Gamma(k-s)\Gamma(k-w)}\sum_{f\in {\mathcal B}_k}\frac{L^*(f,k-s;1/2)L^*(f,k-w)}{\langle f,f\rangle}\\
=&\frac{(2\pi)^{k-w}\Gamma(s+w-k)}{2^{s-1}\Gamma (s)}
\cos(\pi  (s+w-k)/2)\zeta(s+w-k,1/2)\\
&+(-1)^{\frac{k}{2}}\frac{(2\pi)^{k-w}\Gamma(w-s)}{2^{k-s-1}\Gamma (k-s)}
\cos(\pi  (w-s)/2)\zeta(w-s,1/2)\\
&+(-1)^{\frac{k}{2}}\frac{(2\pi)^{k} }{2^{s}\Gamma(s)\Gamma (k-s)}\sum_{a+c/2>0,c>0, (a,c)=1}c^{s-k}(a+c/2)^{-s}\sum_{n>0}n^{w-1}\\
&\times \left(e^{\pi i s/2}e^{2\pi i \frac{na'}{c}}{{}_1{f}_1}\left(s,k;-\frac{2\pi i n }{c(a+c/2)}\right)+e^{-\pi i s/2 }e^{-2\pi i \frac{na'}{c}}{{}_1{f}_1}\left(s,k;\frac{2\pi i  n }{c(a+c/2)}\right)\right)
\end{align*}
holds on $\mathcal D$ and the double sum in the last term of the right-hand side is absolutely convergent on $\mathcal D$.
\begin{proof}
The identity follows easily from Proposition \ref{Prop-2.1} and the formula (\ref{Formula-2.3}), and the absolute convergence on $\mathcal{D}$ follows from that of the series for $E^*_{s,k-s}(z,w;\frac{1}{2})$ (see Remark 4.4 of \cite{CHOIE2020} for the untwisted case).
\end{proof}
\end{Cor}
\section{Analytic continuation}

The left-hand side of the identity in Corollary \ref{Cor-2.2} is meromorphic on ${\mathbb C}^2$, while the right-hand side is only valid on $\mathcal D$. Set
\begin{align*}
\mathcal F &=\{(s, w)\in {\mathbb C}^2: 3/2<\text{Re}(s), \text{Re}(w)<k-2    \}.
\end{align*}
In the following proposition, we will find a new expression in zeta functions and gamma functions that is defined on $\mathcal{F}$ and equal to the right-hand side of Corollary \ref{Cor-2.2} on $\mathcal{F}\cap\mathcal{D}$, so we obtain a new formula for the left-hand side of Corollary \ref{Cor-2.2}.

\begin{Prop}\label{Prop-3.1} We have the following identity on $\mathcal F$
\begin{align*}
&\frac{2^{2-s-w}\pi^{k+1-w}\Gamma(k-1)}{\Gamma(s)\Gamma(k-s)\Gamma(k-w)}\sum_{f\in {\mathcal B}_k}\frac{L^*(f,k-s;1/2)L^*(f,k-w)}{\langle f,f\rangle}\\
=&\frac{(2\pi)^{k-w}\Gamma(s+w-k)}{2^{s-1}\Gamma (s)}
{\cos(\pi(s+w-k)/2)}\zeta(s+w-k,1/2)\\
&+(-1)^{\frac{k}{2}}\frac{(2\pi)^{k-w}\Gamma(w-s)}{2^{k-s-1}\Gamma (k-s)}\cos(\pi(w-s)/2)\zeta(w-s,1/2)\\
&+\frac{(2\pi)^{k-w}\Gamma (w)\Gamma(k-s-w)}{2^{k-s-1}\Gamma (k-s)\Gamma(k-w)}\cos(\pi  (k-s-w)/2)\zeta(k-s-w,1/2)\\
&+(-1)^{\frac{k}{2}}\frac{(2\pi)^{k-w}\Gamma (w) \Gamma(s-w)}{2^{s-1}\Gamma(s)\Gamma(k-w)}\cos(\pi  (s-w)/2)\zeta(s-w,1/2)\\
&+(-1)^{\frac{k}{2}}\frac{(2\pi)^{k-w}\Gamma (w)\Gamma(1-w) }{2^{k-s-1}\Gamma(s)\Gamma(k-s-w+1)}\cos(\pi  (s-w)/2)
{}_2F_1\left[
\begin{matrix}
1-s, & k-s \\
& k-s-w+1
\end{matrix} \bigg| \frac{1}{2}
\right]\\
& +(-1)^{\frac{k}{2}}\frac{(2\pi)^{k-w}\Gamma (w) \Gamma(1-w)}{2^{s-1}\Gamma (k-s)\Gamma(1+s-w)}\cos(\pi  (s+w)/2)
{}_2F_1\left[
\begin{matrix}
s+1-k, & s \\
  &  1+s-w
\end{matrix} \bigg| \frac{1}{2}
\right]\\
&+R(s, w),
\end{align*}
where $R(s,w)$ is holomorphic on $\mathcal F$ and bounded absolutely by a constant multiple of
\begin{align*}
&\frac{|\Gamma (w)| }{|\Gamma(s)\Gamma (k-s)|}e^{\pi (|\text{Im}(s)|+|\text{Im}(w)|)}\zeta(k-1-\max\{\text{Re}(s), \text{Re}(w)\}),
\end{align*}
and ${}_2{F}_1$ is the hypergeometric function that, when $\text{Re}(\gamma)>\text{Re}(\beta)>0,|z|<1$, is given by
\[
{}_2{F}_1\left[
\begin{matrix}
\alpha, & \beta \\
  &  \gamma
\end{matrix} \bigg| z
\right]=\frac{\Gamma(\gamma)}{\Gamma(\beta)\Gamma(\gamma-\beta)}\int_{0}^{1}\frac{t^{\beta-1}(1-t)^{\gamma-\beta-1}}{(1-zt)^{\alpha}}dt.
\]
\end{Prop}
\begin{proof}
To carry out the proof, we only have to deal with the last term of the right-hand side of Corollary \ref{Cor-2.2}, namely
\begin{align*}
A(s,w):=&\frac{(-1)^{\frac{k}{2}}(2\pi)^{k} }{2^{s}\Gamma(s)\Gamma (k-s)}\sum_{a+\frac{c}{2}>0, c>0, (a,c)=1}c^{s-k}\left(a+\frac{c}{2}\right)^{-s}\sum_{n>0}n^{w-1}\\
\times&\left(e^{\pi i s/2}e^{2\pi i na'/c}{{}_1{f}_1}\left(s,k;-\frac{2\pi i n }{c(a+\frac{c}{2})}\right)+e^{-\pi i s/2 }e^{-2\pi i na'/c}{{}_1{f}_1}\left(s,k;\frac{2\pi i  n }{c(a+\frac{c}{2})}\right)\right).
\end{align*}
\textbf{Step I.} We open the integral of ${}_1f_1$, add up the sum on $n$ and utilize Hurwitz zeta function to obtain a new expression for $A(s,w)$ on $\mathcal{D}$. To take care of the analytical issue, set
\[{\mathcal D }_1 =\{(s,w)\in {\mathbb C}^2: 2<\text{Re}(s)<k-2, \ \ \ \text{Re}(w)<0\}.\]
Let $G_{a,c}(s,w):=G'_{a,c}(s,w)+G''_{a,c}(s,w)$, where
\begin{align*}
G'_{a,c}(s,w)=&\sum_{n>0}n^{w-1}e^{\pi i s/2}e^{2\pi i na'/c}{{}_1{f}_1}\left(s,k;-\frac{2\pi i n }{c(a+c/2)}\right),\\
G''_{a,c}(s,w)=&\sum_{n>0}n^{w-1}e^{-\pi i s/2 }e^{-2\pi i na'/c}{{}_1{f}_1}\left(s,k;\frac{2\pi i  n }{c(a+c/2)}\right).
\end{align*}
As a subseries of an absolutely and uniformly convergent series, the function $G'_{a,c}(s,w)$ is holomorphic on $\mathcal D$, so in particular $G_{a,c}(s,w)$ is holomorphic on the smaller region ${\mathcal D}_1$. For $Re(\beta)>Re(\alpha)>0$, it is known that \cite{1997Nonvanishing}
\[
{{}_1{f}_1}(\alpha, \beta; z)=\int_0^1e^{zu}u^{\alpha-1}(1-u)^{\beta-\alpha-1}du.
\]
On ${\mathcal D}_1$, we have
\begin{align*}
G'_{a,c}(s,w)=&\sum_{n>0}n^{w-1}e^{2\pi i na'/c}e^{\pi i s/2}{{}_1{f}_1}\left(s,k;-\frac{2\pi i n }{c(a+c/2)}\right)\\
=&e^{\pi i s/2}\int_0^1 u^{s-1}(1-u)^{k-s-1}\sum_{n>0}n^{w-1}e^{2\pi i na'/c}e^{-\frac{2\pi i n }{c(a+c/2)}u}du,
\end{align*}
where the interchange of summation and integration is justified because of absolute convergence on ${\mathcal D}_1$ (but not on $\mathcal{D}$!). On ${\mathcal D}_1$,
\begin{align*}
G'_{a,c}(s,w)=e^{\pi i s/2}\int_0^1 u^{s-1}(1-u)^{k-s-1}F\left(1-w,\frac{a'}{c}-\frac{u}{(a+c/2)c}\right)du.
\end{align*}
To switch to Hurwitz zeta function, we have to replace the parameter $\alpha$ in $F(s, \alpha)$ by $\{\alpha\}$. For $0<\{\alpha\}<1$, apply Hurwitz zeta function (see Formula 25.13.2 of \cite{Olver2010NIST}), we have
\begin{align*}
G'_{a,c}(s,w)=&(2\pi)^{-w}\Gamma (w)\int_0^{1} u^{s-1}(1-u)^{k-s-1}\\
&\times \left(e^{\pi i(s+w)/2}\zeta\left(w,\left\{\alpha_{a,c}(u)\right\}\right)+e^{\pi i (s-w)/2}\zeta\left(w,1-\left\{\alpha_{a,c}(u)\right\}\right)\right)du,
\end{align*}
where for ease of notation we set $\alpha_{a,c}(u)=\frac{a'}{c}-\frac{u}{(a+c/2)c}$.
Similarly,
\begin{align*}
G''_{a,c}(s,w)=&(2\pi)^{-w}\Gamma (w)\int_0^{1} u^{s-1}(1-u)^{k-s-1}\\
&\times\left(e^{-\pi i (s+w)/2}\zeta\left(w,\left\{\alpha_{a,c}(u)\right\}\right)+e^{-\pi i (s-w)/2}\zeta\left(w,1-\left\{\alpha_{a,c}(u)\right\}\right)\right)du.
\end{align*}
So we have, on ${\mathcal D}_1$
\begin{equation}\label{Formula-3.1}
\begin{aligned}
&G_{a,c}(s,w)\\
=&2(2\pi)^{-w}\Gamma (w)\int_0^{1} u^{s-1}(1-u)^{k-s-1}\\
&\times\left(\cos(\pi  (s+w)/2)\zeta\left(w,\left\{\alpha_{a,c}(u)\right\}\right)+\cos(\pi  (s-w)/2)\zeta\left(w,1-\left\{\alpha_{a,c}(u)\right\}\right)\right)du.
\end{aligned}
\end{equation}
Note that the right-hand side of (\ref{Formula-3.1}) is meromorphic on $\mathcal D$, forcing (\ref{Formula-3.1}) to hold on $\mathcal D$ as well since $G_{a,c}(s,w)$ is holomorphic on  $\mathcal D$.
Replacing the expression $G_{a,c}(s,w)$ in $A(s,w)$ with the right-hand side of (\ref{Formula-3.1}), the following equality
\begin{equation}\label{Formula-3.2}
\begin{aligned}
&A(s,w)\\
=&(-1)^{\frac{k}{2}}\frac{(2\pi)^{k-w}\Gamma (w) }{2^{s-1}\Gamma(s)\Gamma (k-s)}\sum_{a+c/2>0, c>0, (a,c)=1}c^{s-k}(a+c/2)^{-s}\int_0^{1} u^{s-1}(1-u)^{k-s-1}\\
&\times (\cos(\pi(s+w)/2))\zeta\left(w,\left\{\alpha_{a,c}(u)\right\}\right)+\cos(\pi  (s-w)/2)\zeta\left(w,1-\left\{\alpha_{a,c}(u)\right\}\right)du
\end{aligned}
\end{equation}
holds on $\mathcal D$.\\
\textbf{Step II.} Next we remove subseries from that of $A(s,w)$, which give functions meromorphic everywhere, and then prove that the remaining series is absolutely convergent on $\mathcal{F}$. Split the series of \eqref{Formula-3.2} into four subseries $A(s,w)_i$ and we treat them one by one.
\\
$(\textbf{1})$
For the terms with $c=1$ and $0<a+c/2<1$, obviously $c=1$ and $a=0$. We ignore the factors $(-1)^{\frac{k}{2}}\frac{(2\pi)^{k-w}\Gamma (w) }{2^{s-1}\Gamma(s)\Gamma (k-s)}$ for now and separate the first term in the Hurwitz zeta functions, so we have the following four terms
\begin{align*}
&2\int_0^{1} u^{s-1}(1-u)^{k-s-1}\left(\cos(\pi  (s+w)/2)(\{1-2u\})^{-w}\right)du\\
+&2\int_0^{1} u^{s-1}(1-u)^{k-s-1}\left(\cos(\pi  (s-w)/2)(\{2u\})^{-w}\right)du\\
+&2\int_0^{1} u^{s-1}(1-u)^{k-s-1}\left(\cos(\pi  (s+w)/2)\zeta(w,1+\{1-2u\})\right)du\\
+&2\int_0^{1} u^{s-1}(1-u)^{k-s-1}\left(\cos(\pi  (s-w)/2)\zeta(w,1+\{2u\})\right)du.
\end{align*}
The third and the fourth term are bounded absolutely by a constant multiple of
\[e^{\pi(|\text{Im}(s)|+|\text{Im}(w)|)},\]
hence giving holomorphic functions on $\mathcal F$. For the second term, we split the interval into two halves $[0,1/2]$ and $[1/2,1]$ to simplify $\{2u\}$. Then by writing
$\int_{0}^{1/2}=\int_{0}^{1}-\int_{1/2}^{1}=\int_{0}^{1}+R(s,w)_{1}'$
and a change of variable $v=2(1-u)$ for $\int_{1/2}^1$,  the second term is equal to
\begin{equation*}
\begin{aligned}
&\frac{\Gamma(s-w)\Gamma(k-s)}{2^{w}\Gamma(k-w)}\cos(\pi  (s-w)/2)+R(s,w)_{1}'\\
+&\frac{\Gamma(k-s)\Gamma(1-w)}{2^{k-s}\Gamma(k-s-w+1)}\cos(\pi  (s-w)/2){}_2F_1\left[
\begin{matrix}
1-s, & k-s \\
& k-s-w+1
\end{matrix} \bigg| \frac{1}{2}
\right],
\end{aligned}
\end{equation*}
where $R(s,w)_{1}'$ is a holomorphic function on $\mathcal F$ and bounded by $e^{\pi(|Im(w)|+|Im(s)|)}$ up to a constant.
The first term can be treated similarly and is equal to
\begin{equation*}
\begin{aligned}
&\frac{\Gamma(s)\Gamma(1-w)}{2^{s}\Gamma(1+s-w)}\cos(\pi  (s+w)/2){}_2F_1\left[
\begin{matrix}
s+1-k, & s \\
  &  1+s-w
\end{matrix} \bigg| \frac{1}{2}
\right]\\
+&\frac{\Gamma(s)\Gamma(k-s-w)}{2^{w}\Gamma(k-w)}\cos(\pi  (s+w)/2)+R(s,w)_{1}'',
\end{aligned}
\end{equation*}
where $R(s,w)_{1}''$ is a holomorphic function on $\mathcal F$ and bounded by $e^{\pi(|Im(w)|+|Im(s)|)}$ up to a constant.
Combining these formulas on the four terms, we have the full contribution of such terms
\begin{align*}
&A(s,w)_{1}\\
=&\frac{(2\pi)^{k-w}\Gamma (w) \Gamma(k-s-w)}{2^{w-1}\Gamma (k-s)\Gamma(k-w)}\cos(\pi (k-s-w)/2)\\
&+(-1)^{\frac{k}{2}}\frac{(2\pi)^{k-w}\Gamma (w) \Gamma(s-w)}{2^{w-1}\Gamma(s)\Gamma(k-w)}\cos(\pi  (s-w)/2)\\
&+\frac{(-1)^{\frac{k}{2}}(2\pi)^{k-w}\Gamma (w) \Gamma(1-w)}{2^{s-1}\Gamma (k-s)\Gamma(1+s-w)}\cos(\pi (s+w)/2){}_2F_1\left[
\begin{matrix}
s+1-k, & s \\
  &  1+s-w
\end{matrix} \bigg| \frac{1}{2}
\right]\\
&+\frac{(-1)^{\frac{k}{2}}(2\pi)^{k-w}\Gamma (w) \Gamma(1-w)}{2^{k-s-1}\Gamma(s)\Gamma(k-s-w+1)}\cos(\pi  (s-w)/2){}_2F_1\left[
\begin{matrix}
1-s, & k-s \\
& k-s-w+1
\end{matrix} \bigg| \frac{1}{2}
\right]\\
&+R(s,w)_{1},
\end{align*}
where $R(s,w)_{1}$ is a holomorphic function on $\mathcal F$ and is bounded absolutely by
\[
\frac{|\Gamma (w)| }{|\Gamma(s)\Gamma (k-s)|}e^{\pi(|\text{Im}(w)|+|\text{Im}(s)|)}\] up to a constant.
\\
$(\textbf{2})$ For the terms with  $ 0<a+c/2<1,  c>1 $, we have
\[
a+c/2=1/2,\quad c=1+2m,\quad a=-m,\quad m\ge1,\quad a'=2.
\]
And then $0\le\frac{a'}{c}-\frac{u}{(a+c/2)c}<1$, $1/3\le1-\frac{a'}{c}+\frac{u}{(a+c/2)c}<1$. A similar but easier treatment gives the corresponding contribution $A(s,w)_2$:
\begin{align*}
&A(s,w)_{2}\\
=&\frac{(2\pi)^{k-w}\Gamma (w)\Gamma(k-s-w)}{2^{k-s-1}\Gamma (k-s)\Gamma(k-w)}\cos(\pi (k-s-w)/2)\zeta(k-s-w,3/2)+R(s,w)_{2},
\end{align*}
where $R(s,w)_{2}$ is holomorphic on $\mathcal F$ and bounded absolutely by
$
\frac{|\Gamma (w)|}{|\Gamma(s)\Gamma (k-s)|}e^{\pi (|\text{Im} (s)|+|\text{Im}(w)|)}$ up to a constant.
\\
$(\textbf{3})$
Similarly, the terms with $c=1$ and $a+c/2\ge1$ contribute
\begin{align*}
&A(s,w)_{3}=(-1)^{\frac{k}{2}}\frac{(2\pi)^{k-w}\Gamma (w)\Gamma(s-w) }{2^{s-1}\Gamma(s)\Gamma(k-w)}\cos(\pi (s-w)/2)\zeta(s-w,3/2)+R(s,w)_{3},
\end{align*}
where $R(s,w)_{3}$ is holomorphic on $\mathcal F$ and bounded absolutely by $
\frac{|\Gamma (w)|}{|\Gamma(s)\Gamma (k-s)|}e^{\pi (|\text{Im} (s)|+|\text{Im}(w)|)}$ up to a constant.
\\
(\textbf {4})  We are left with the terms with $a+c/2 \ge 1, c>1$ and we prove that the corresponding contribution $A(s,w)_4=R(s,w)_4$ is itself holomorphic on $\mathcal{F}$. 

We first treat the terms with $\zeta(w,1-\frac{a'}{c}+\frac{u}{(a+c/2)c})$.
Note that for each pair $(a,c)$, we always have $0<\frac{a'}{c}-\frac{u}{(a+c/2)c}<1$, so for $\text{Re}(w)>3/2$
\begin{equation*}\label{For-3.1}
\begin{aligned}
&\left|\zeta\left(w,1-\frac{a'}{c}+\frac{u}{c(a+c/2)}\right)\right|\le \left|\left(1-\frac{a'}{c}+\frac{u}{c(a+c/2)}\right)^{-w}\right|+\zeta(\text{Re}(w))\\
&<\left(\frac{c-a'}{c}+\frac{u}{(a+c/2)c}\right)^{-\text{Re}(w)}+\zeta(3/2)\le \left(\frac{c}{c-a'}\right)^{\text{Re}(w)}+\zeta(3/2)\ll \left(\frac{c}{c-a'}\right)^{\text{Re}(w)}.
\end{aligned}
\end{equation*}
Then we split such pairs $(a,c)$ into two parts. For $1\le a+c/2< c$, note that $1\le a'\le c-1, 1\le a+c/2\le c-1/2, aa'\equiv 1\mod c$. Therefore $(c-a')(a+c/2)\ge c/2-1$. For $c\ge4 $, then $c/2-1\ge c/4$, so the series over such pairs $(a, c)$ is bounded absolutely by a constant multiple of
\begin{equation*}
\begin{aligned}
 &e^{\pi (|\text{Im}(s)|+|\text{Im}(w)|)}\sum_{c=2}^{\infty}\sum_{ \substack{1\le a+c/2< c\\ (a,c)=1 }}c^{-k+\text{Re}(s)+\text{Re}(w)}(a+c/2)^{-\text{Re}(s)}(c-a')^{-\text{Re}(w)}\\
  \ll&e^{\pi (|\text{Im}(s)|+|\text{Im}(w)|)}\sum_{c=4}^{\infty}\sum_{ \substack{1\le a+c/2< c\\ (a,c)=1 }} c^{-k+\text{Re}(s)+\text{Re}(w)}((a+c/2)(c-a'))^{-\min\{\text{Re}(s), \text{Re}(w)\}}\\
    \ll&4^{\min\{\text{Re}(s), \text{Re}(w)\}}e^{\pi (|\text{Im}(s)|+|\text{Im}(w)|)}\sum_{c=4}^{\infty}c\cdot c^{-k+\max\{\text{Re}(s), \text{Re}(w)\}}\\
     \ll&e^{\pi (|\text{Im}(s)|+|\text{Im}(w)|)}\zeta(k-1-\max\{\text{Re}(s), \text{Re}(w)\}).
\end{aligned}
\end{equation*}
For $a+c/2\ge c$, since $(\frac{c}{c-a'})^{\text{Re}(w)}<c^{\text{Re}(w)}$, the series over such pairs $(a, c)$ is absolutely bounded by $e^{\pi (|\text{Im}(s)|+|\text{Im}(w)|)}\zeta(k-\text{Re}(w)-1)$ up to a constant. Combining the above two bounds, the subseries in this case is bounded absolutely by a constant multiple of
\begin{align*}
&e^{\pi (|\text{Im}(s)|+|\text{Im}(w)|)}\zeta(k-1-\max\{\text{Re}(s), \text{Re}(w)\}),
\end{align*}
and it give a holomorphic function on $\mathcal F$.

Next we treat the terms with $\zeta(w,\frac{a'}{c}-\frac{u}{(a+c/2)c})$. Note first that $a'$ and $a+c/2$ cannot both be equal to $1$, since otherwise $aa'-b_0c=1=a-b_0c$ and $b_0=-1/2$ which is absurd. It follows that
 \[
 1>\frac{a'}{c}-\frac{u}{(a+c/2)c}\ge \frac{a'}{c}-\frac{1}{(a+c/2)c}\ge\frac{1}{c}\left(1-\frac{1}{3/2}\right)=\frac{1}{3c},
 \]
and with such a lower bound we can proceed similarly as above. Namely, split such pairs $(a,c)$ into two parts $a+c/2\ge c$ and $1\le a+c/2< c$, and the series with terms subject to $a+c/2\ge c$ is bounded absolutely by a constant multiple of
\begin{equation*}
e^{\pi (|\text{Im}(s)|+|\text{Im}(w)|)}\zeta(k-\text{Re}(w)-1)
\end{equation*} and that for $1\le a+c/2< c$ is bounded absolutely by a constant multiple of
\begin{equation*}
e^{\pi (|\text{Im}(s)|+|\text{Im}(w)|)}\zeta(k-1-\max\{\text{Re}(s), \text{Re}(w)\}).
\end{equation*}
So the series over such pairs $(a, c)$ are bounded absolutely by a constant multiple of
\begin{align*}
&e^{\pi (|\text{Im}(s)|+|\text{Im}(w)|)}\zeta(k-1-\max\{\text{Re}(s), \text{Re}(w)\}).
\end{align*}
Therefore $A(s,w)_{4}$ is itself a holomorphic function on $\mathcal F$ and it is bounded absolutely by a constant multiple of
\begin{align*}
&\frac{|\Gamma (w)| }{|\Gamma(s)\Gamma (k-s)|}e^{\pi (|\text{Im}(s)|+|\text{Im}(w)|)}\zeta(k-1-\max\{\text{Re}(s), \text{Re}(w)\}).
\end{align*}

Combining all of $A(s,w)_i$ and applying the equation $\zeta(s, 3/2)+2^{s}=\zeta(s,1/2)$, we obtain the main terms of $A(s,w)$. The bound of the remainder $R(s,w)$ is obvious by combining that of $R(s,w)_i$. So we end the proof.
\end{proof}
\section{The Rationality}
To prove the rationality, we recall the following formula on special values of the hypergeometric function ${}_2F_1$.  For integers $a,b,n$, we have
\begin{align}
\label{Formula-4.1}{}_2F_1\left[
\begin{matrix}
a, & b \\
& \frac{1}{2}(a+b+n+1)
\end{matrix} \bigg| \frac{1}{2}
\right]
=&\frac{\Gamma(\frac{1}{2})\Gamma(\frac{1}{2}(a+b+n+1))}{\Gamma(\frac{b}{2})\Gamma(\frac{1}{2}(b+1))}\frac{\Gamma(\frac{1}{2}(a-b-|n|+1))}{\Gamma(\frac{1}{2}(a-b+n+1))}
\\ &\times \sum_{r=0}^{|n|}\binom{|n|}{r} \frac{\delta(n,r)\Gamma(\frac{1}{2}(b+r))}{\Gamma(\frac{1}{2}(a-|n|+r+1))},\nonumber
\end{align}
where
\[
\delta(n,r):=\left\{\begin{matrix}
   (-1)^r & n\geq 0  \\
   1 & n<0  \\
\end{matrix}\right.. 
\]
Note that formula \eqref{Formula-4.1} combines Theorem 1 and Theorem 2 of \cite{2011Generalizations}.

For our purpose, we only need the first Fourier coefficient of the twisted double Eisenstein series.
\begin{Prop}\label{Prop-4.1} For integers $s, w$ of opposite parity with $(s, w)\in \mathcal F$,
\begin{align*}
&2(2\pi)^{w-k-1}c(1)\\
=&\frac{2^{s+w-k}-1}{2^{s}\Gamma(s)}
\rho(k-w-s+1)+(-1)^{\frac{k}{2}}\frac{2^{w-s}-1}{2^{k-s}\Gamma(k-s)}\rho(s-w+1)\\
&+\frac{(2^{k-s-w}-1)\Gamma(w)}{2^{k-s}\Gamma(k-s)\Gamma(k-w)}\rho(s+w-k+1)+(-1)^{\frac{k}{2}}\frac{(2^{s-w}-1)\Gamma(w)}{2^{s}\Gamma(s)\Gamma(k-w)}\rho(w-s+1)\\
&+
\frac{(-1)^{(|k-2w|+w- s-1)/2}\Gamma (w) }{2\Gamma(s)\Gamma(k-s)\Gamma((|k-2w|+k)/2)}\\
 &\quad\times  \sum_{r=0}^{|k-2w|}\delta(k-2w,r)\binom{|k-2w|}{r}\prod_{j=1}^{(k+|k-2w|-2)/2}(j+(-s-|k-2w|+r)/2)\\
& +
 \frac{(-1)^{(|k-2w|+w+s-1)/2}\Gamma (w)}{2\Gamma(s)\Gamma(k-s)\Gamma((|k-2w|+k)/2)}\\
 &\quad\times  \sum_{r=0}^{|k-2w|}\delta(k-2w,r)\binom{|k-2w|}{r}\prod_{j=1}^{(k+|k-2w|-2)/2}(j+(s-k-|k-2w|+r)/2),
\end{align*}
where
\[
\rho(2n):=\left\{\begin{matrix}
   (-1)^{n+1}B_{2n}/(2n)! & n\ge0,  \\
   0 & n<0.  \\
\end{matrix}\right. \ \ \ \ \ \
\]
\end{Prop}
\begin{proof}
By \eqref{Formula-2.3}, Proposition \ref{Prop-3.1} and  \eqref{Formula-4.1}, we have
\[
c(1)=c(1)'+c(1)''+c(1)'''.
\]
where
\begin{align*}
c(1)'=&\frac{(2\pi)^{k-w}\Gamma(s+w-k)}{2^{s-1}\Gamma (s)}
{\cos(\pi(s+w-k)/2)}\zeta(s+w-k, 1/2)\\
&\quad+(-1)^{\frac{k}{2}}\frac{(2\pi)^{k-w}\Gamma(w-s)}{2^{k-s-1}\Gamma (k-s)}\cos(\pi(w-s)/2)\zeta(w-s, 1/2)\\
&\quad+\frac{(2\pi)^{k-w}\Gamma (w)\Gamma(k-s-w)}{2^{k-s-1}\Gamma (k-s)\Gamma(k-w)}\cos(\pi (k-s-w)/2)\zeta(k-s-w, 1/2)\\
&\quad+(-1)^{\frac{k}{2}}\frac{(2\pi)^{k-w}\Gamma (w)\Gamma(s-w) }{2^{s-1}\Gamma(s)\Gamma(k-w)}\cos(\pi (s-w)/2)\zeta(s-w, 1/2),
\end{align*}
\begin{align*}
c(1)''=&(-1)^{\frac{k}{2}}
\frac{(2\pi)^{k-w}\Gamma(1/2)\Gamma (w) \Gamma((2-k-|k-2w|)/2)\cos(\pi (s-w)/2)}{2^{k-s-1}\Gamma(s)\Gamma((k-s)/2)\Gamma((k-s+1)/2)}\\
 &\quad\times  \sum_{r=0}^{|k-2w|}\binom{|k-2w|}{r} \frac{\delta(k-2w,r)\Gamma((k-s+r)/2)}{\Gamma((2-s-|k-2w|+r)/2)} \\
& +(-1)^{\frac{k}{2}}
 \frac{(2\pi)^{k-w}\Gamma(1/2)\Gamma (w)\Gamma((2-k-|k-2w|)/2) \cos(\pi (s+w)/2)}{2^{s-1}\Gamma(k-s)\Gamma(s/2)\Gamma((s+1)/2)}\\
 &\quad\times  \sum_{r=0}^{|k-2w|}\binom{|k-2w|}{r} \frac{\delta(k-2w,r)\Gamma((s+r)/2)}{\Gamma((2-k+s-|k-2w|+r)/2)},
\end{align*}
and $c(1)'''=R(s,w)$.

For the term $c(1)'$, we have to compute the values of $ \Gamma(s)\zeta(s, 1/2)\cos(\pi s/2)$ at odd integers $s=m$. Recall that
\[\Gamma(s)\zeta(s, 1/2)\cos(\pi s/2)=(2^s-1)\Gamma(s)\zeta(s)\cos(\pi s/2).\] If $m=1$, by collecting the residue of $\zeta(s, 1/2)$ and the first Taylor coefficient of the rest at $s=1$, the value is equal to $-\pi/2$.
It is clear that if $m>1$, it has value $0$ since all factors are holomorphic at $m$.  Finally, if $m<0$, recall that
\[\textrm{Res}(\Gamma(\cdot), m)=\frac{(-1)^{-m}}{(-m)!},\quad \zeta(m)=(-1)^{-m}\frac{B_{1-m}}{1-m},\] where $B_{n}$ is $n$-th Bernoulli number. Therefore, the value we seek is
$(-1)^{\frac{m+1}{2}}\frac{(2^m-1)\pi B_{1-m}}{2(1-m)!}$ and the formula of $c(1)'$ in the statement follows.

For the term $c(1)''$, note that 
\begin{align*}
&\Gamma\left(\frac{2-k-|k-2w|}{2}\right)\cos\left(\pi\frac{w\pm s}{2} \right)=\frac{(-1)^{(|k-2w|+k+w\pm s-1)/2}}{\Gamma((|k-2w|+k)/2)}\frac{\pi}{2},
\end{align*}
so the formula of $c(1)''$ in the statement follows from the following identities for $\beta=s$ and for $\beta=k-s$:
\begin{align*}
\Gamma(\beta/2)\Gamma((\beta+1)/2)&=2^{1-\beta}\Gamma(1/2)\Gamma(\beta),\\
\frac{\Gamma((\beta+r)/2)}{\Gamma((2-k+\beta-|k-2w|+r)/2)}&=\prod_{j=1}^{(k+|k-2w|-2)/2}(j+(\beta-k-|k-2w|+r)/2).
\end{align*}

For the last term $c(1)'''$, by the proof of Proposition \ref{Prop-3.1}, where we actually proved that the resulting subseries in $(a,c)$ with $\cos(\pi(s+w)/2)$ and $\cos(\pi(s-w)/2)$ removed are absolutely convergent. Then plug in the values $w,s$ and we see that $c(1)'''=0$ because
\[
\cos(\pi (s+w)/2)=0, \ \ \cos(\pi(s-w)/2)=0.
\]
This completes the proof.
\end{proof}

Now we are ready to prove the following formula on the Petersson inner product of the Cohen kernels, from which our main theorem in the introduction follows trivially when $(s,w)\in\mathcal{F}$. If $s$ or $w=k-2$, by applying the functional equation $(s,w)\rightarrow (k-s,w)$ or $(s,w)\rightarrow (s,k-w)$ to send the point into $\mathcal{F}$, the same rationality holds.

\begin{Thm} \label{Thm2.3} For integers $s, w$ of opposite parity with $(s, w)\in \mathcal F$,
\begin{align}\label{Formula-4.2}
\nonumber&\pi^{-2} e^{\pi i (s-w)/2}\frac{2^{k-2}\Gamma(w)}{2^{s}\Gamma(k-1)}\langle{\mathcal{C}}_{k}(z, s; 1/2),{\mathcal{C}}_{k}(z, w)\rangle\\
\nonumber=&\frac{2^{s+w-k}-1}{2^{s}\Gamma(s)}\rho(k-w-s+1)+(-1)^{\frac{k}{2}}\frac{(2^{w-s}-1)}{2^{k-s}\Gamma(k-s)}\rho(s-w+1)\\
\nonumber&+\frac{(2^{k-s-w}-1)\Gamma(w)}{2^{k-s}\Gamma(k-s)\Gamma(k-w)}\rho(s+w-k+1)+(-1)^{\frac{k}{2}}\frac{(2^{s-w}-1)\Gamma(w)}{2^{s}\Gamma(s)\Gamma(k-w)}\rho(w-s+1)\\
&+
\frac{(-1)^{(|k-2w|+w- s-1)/2}\Gamma (w) }{2\Gamma(s)\Gamma(k-s)\Gamma((|k-2w|+k)/2)}\\
 \nonumber&\quad\times  \sum_{r=0}^{|k-2w|}\delta(k-2w,r)\binom{|k-2w|}{r}\prod_{j=1}^{(k+|k-2w|-2)/2}(j+(-s-|k-2w|+r)/2)\\
\nonumber& +
 \frac{(-1)^{(|k-2w|+w+s-1)/2}\Gamma (w)}{2\Gamma(s)\Gamma(k-s)\Gamma((|k-2w|+k)/2)}\\
 \nonumber&\quad\times  \sum_{r=0}^{|k-2w|}\delta(k-2w,r)\binom{|k-2w|}{r}\prod_{j=1}^{(k+|k-2w|-2)/2}(j+(s-k-|k-2w|+r)/2).
 \end{align}
\end{Thm}
\begin{proof}
When $(s, w)\in \mathcal F$ and $f\in {\mathcal B}_k$, by equation (3-3) of \cite{2013Kernels}, we have
\begin{align*}
\langle {\mathcal{C}}_{k}(\cdot, s; 1/2), f\rangle  &=2^{2-k}\pi e^{-\pi i s/2}\frac{\Gamma(k-1)}{\Gamma(s)\Gamma(k-s)}L^*(f, k-s;1/2),\\
\langle {\mathcal{C}}_{k}(\cdot, w; 0), f\rangle  &=2^{2-k}\pi e^{-\pi i w/2}\frac{\Gamma(k-1)}{\Gamma(w)\Gamma(k-w)}L^*(f, k-w).
\end{align*}
Therefore
 \begin{align*}
{\mathcal{C}}_{k}(z, s; 1/2)&=2^{2-k}\pi e^{-\pi i s/2}\frac{\Gamma(k-1)}{\Gamma(s)\Gamma(k-s)}\sum_{f\in{\mathcal{B}}_{k}}L^*(f, k-s;1/2)\langle f,f\rangle  ^{-1}f(z),\\
{\mathcal{C}}_{k}(z, w)&=2^{2-k}\pi e^{-\pi i w/2}\frac{\Gamma(k-1)}{\Gamma(w)\Gamma(k-w)}\sum_{f\in{\mathcal{B}}_{k}}L^*(f, k-w)\langle g,g\rangle  ^{-1}g(z).
\end{align*}
Then by \eqref{Formula-2.3} we have
 \begin{equation}\label{Formula-4.3}
 \begin{aligned}
&\langle {\mathcal{C}}_{k}(z, s; 1/2),{\mathcal{C}}_{k}(z, \overline{w})\rangle  \\
=&\frac{2^{2(2-k)}\pi^2 e^{\pi i (w-s)/2}\Gamma(k-1)^2}{\Gamma(s)\Gamma(k-s)\Gamma(w)\Gamma(k-w)}
\sum_{f\in{\mathcal{B}}_{k}}L^*(f, k-s;1/2)L^*(f, k-w)\langle f,f\rangle  ^{-1}\\
=&\frac{2^{2(2-k)}\pi^2 e^{\pi i (w-s)/2}\Gamma(k-1)^2}{\Gamma(s)\Gamma(k-s)\Gamma(w)\Gamma(k-w)}\cdot \frac{\Gamma(s)\Gamma(k-s)\Gamma(k-w)}{2^{2-s-w}\pi^{k+1-w}\Gamma(k-1)}c(1).
 \end{aligned}
\end{equation}
For integers $s,w$ of opposite parity in $\mathcal{F}$, by Proposition \ref{Prop-4.1}, we obtain the formula (\ref{Formula-4.2}). 
\end{proof}

\begin{Rmk} For general rational parameters $\frac{p}{q}$, we can also compute the Fourier coefficients of $E_{s,k-s}^{*}(z,w; \frac{p}{q})$ and obtain their analytic continuation along the same lines without much trouble. However, we lack information on ${}_2F_1\left[
\begin{matrix}
a, & b \\
  &  c
\end{matrix} \bigg| \frac{p}{q}
\right]$  for $q\neq 1,2$ and Theorem \ref{Thm2.3} is the best we can do at present time.
\end{Rmk}

 \bibliographystyle{amsplain}
\bibliography{paper}

\end{document}